\documentclass[11pt,reqno,a4paper]{amsart}
\usepackage{color,enumerate,cancel,hyperref,mathtools,tikz-cd,amssymb,amsfonts}
\newtheorem{theorem}{Theorem}
\newtheorem{lemma}[theorem]{Lemma}
\newtheorem{proposition}[theorem]{Proposition}
\newtheorem{corollary}[theorem]{Corollary}

\theoremstyle{definition}

\newtheorem{example}[theorem]{Example}
\newtheorem*{problem}{Problem}

\newcommand{\Lip}{\operatorname{Lip}_0}

\newcommand{\R}{\mathbb{R}}
\newcommand{\N}{\mathbb{N}}

\newcommand{\F}{\mathcal{F}}

\theoremstyle{remark}
\newtheorem{remark}[theorem]{Remark}
\numberwithin{equation}{section}

\begin{document}

\title[A lifting theorem for operators between Lipschitz spaces]{A lifting theorem for operators between Lipschitz spaces}
\author[L. Candido]{Leandro Candido}
\address{Universidade Federal de S\~ao Paulo - UNIFESP. Instituto de Ci\^encia e Tecnologia. Departamento de Matem\'atica. S\~ao Jos\'e dos Campos - SP, Brasil}
\email{leandro.candido@unifesp.br}

\thanks{The author is supported by Funda\c c\~ao de Amparo \`a Pesquisa do Estado de S\~ao Paulo - FAPESP No. 2023/12916-1 }


 
\begin{abstract}
We prove that every bounded linear operator between Lipschitz spaces  admits a lifting along the de Leeuw embedding. More precisely, given pointed metric spaces $M$ and $N$ and $\epsilon>0$, every bounded linear operator $S:\mathrm{Lip}_0(M)\to \mathrm{Lip}_0(N)$ admits a lifting $\mathfrak{S}:C(\beta \widetilde{M})\to C(\beta \widetilde{N})$ such that $\|\mathfrak{S}\|\leq \|S\|+\epsilon$ and $\mathfrak{S}(\varPhi_M(f))=\varPhi_N(S(f))$ for every $f\in \mathrm{Lip}_0(M)$. Moreover, compact operators admit compact liftings. 
\end{abstract}

\maketitle

\section{Introduction}
In this work we study spaces of real-valued Lipschitz functions defined on a metric space $M$ that vanish at a fixed distinguished point $0$, denoted by $\Lip(M)$. These spaces, together with their canonical preduals known as Lipschitz-free spaces $\F(M)$, have attracted considerable attention in nonlinear functional analysis; see, for instance, \cite{1111}, \cite{2222}, \cite{3333}, \cite{4444}, \cite{5555}, \cite{CDW}.

A useful approach in the study of $\Lip(M)$ is to consider its canonical isometric embedding into a space of continuous functions. More precisely, let $\widetilde{M}=\{(x,y)\in M\times M:x\neq y\}$ be equipped with a product metric, and let $\beta \widetilde{M}$ denote its Stone-\v{C}ech compactification. The de Leeuw map $\varPhi_M:\Lip(M)\to C(\beta \widetilde{M})$ (see \cite[\S 2.4]{Weaver}) is the linear isometric embedding defined by extending continuously to $\beta \widetilde{M}$ the function
\[
(x,y)\mapsto \frac{f(x)-f(y)}{d(x,y)}.
\]

This embedding allows one to study $\Lip(M)$ within the framework of $C(K)$ spaces, making available techniques from the theory of spaces of continuous functions. However, the topology of $\beta\widetilde M$ is typically quite intricate.

A natural question is whether operators acting on Lipschitz spaces can also be represented at the level of the corresponding spaces of continuous functions.

The purpose of this paper is to show that bounded linear operators between Lipschitz spaces are compatible with the de Leeuw embedding up to an arbitrarily small perturbation in the norm. More precisely, we prove that every bounded linear operator $S:\Lip(M)\to \Lip(N)$ can be lifted to an operator $\mathfrak{S}:C(\beta \widetilde{M})\to C(\beta \widetilde{N})$ such that the following diagram commutes:
    \begin{center}
        \begin{tikzcd}
            & \Lip(M) \arrow[hookrightarrow]{d}{\varPhi_M}\arrow{r}{S} & \Lip(N) \arrow[hookrightarrow]{d}{\varPhi_N} \\
            & C(\beta \widetilde{M}) \arrow[dashed]{r}{\mathfrak{S}}                &C(\beta \widetilde{N})
        \end{tikzcd}
    \end{center}

More specifically, our main result is as follows:

\begin{theorem}\label{Main1}
Let $M$ and $N$ be pointed metric spaces and let $\epsilon>0$ be arbitrary. Every bounded linear operator $S:\Lip(M)\to\Lip(N)$ can be lifted to an operator $\mathfrak{S}:C(\beta \widetilde{M})\to C(\beta \widetilde{N})$ with $\|\mathfrak{S}\|\leq \|S\|+\epsilon$ and such that 
\[\mathfrak{S}(\varPhi_M(f))=\varPhi_N(S(f))\]
for every $f\in \Lip(M)$.
If in particular $S$ is a compact operator, then the lifting $\mathfrak{S}$ can be chosen to be a compact operator.
\end{theorem}

Theorem \ref{Main1} allows bounded linear operators on Lipschitz spaces to be represented, up to an arbitrarily small increase of the norm, by operators acting on suitable spaces of continuous functions. Furthermore, compact operators admit compact liftings.

The paper is organized as follows. Section \ref{SecBasic} contains the necessary preliminaries. In Section \ref{SecLifting} we prove Theorem \ref{Main1}. Finally, in Section \ref{SecIso} we discuss the existence of liftings of isomorphisms.

%------------------------------------------------------------------------------------------------
%------------------------------------------------------------------------------------------------
%------------------------------------------------------------------------------------------------

\section{Basic Terminology}
\label{SecBasic}

A metric space $(M,d)$ is said to be pointed if it is endowed with a fixed distinguished point $0$, called the origin. For a pointed metric space $M$, 
$\Lip(M)$ denotes the Banach space of all Lipschitz functions $f:M\to \R$ such that $f(0)=0$, equipped with the norm
\[
\|f\|_{\Lip} = \sup_{x\neq y \in M} \frac{|f(x)-f(y)|}{d(x,y)}.
\]

The Banach space $\Lip(M)$ is a dual space whose canonical predual, denoted by $\F(M)$ and called the Lipschitz-free space (or Arens--Eells space), is defined as the closed linear span of all normalized elementary molecules in $\Lip(M)^*$. For every $x, y \in M$, the normalized elementary molecule associated to $(x, y)$ is the functional $\overline{m}_{xy}: \Lip(M) \to \R$ given by
\[
\langle \overline{m}_{xy},f\rangle=\frac{f(x)-f(y)}{d(x,y)}.
\]
For a detailed exposition on this topic we refer the reader to \cite{Weaver}.
 
If $K$ is a Hausdorff compact space, $C(K)$ denotes the Banach space of all real-valued continuous functions on $K$, endowed with the supremum norm. The dual space $C(K)^*$ is identified, via the Riesz representation theorem, with the space $\mathcal{M}(K)$ of all signed Radon measures on $K$ of bounded variation, equipped with the variation norm. For every 
$t\in K$, $\delta_t\in \mathcal{M}(K)$ denotes the Dirac measure concentrated at $t$. For further details concerning Banach spaces of the form $C(K)$ see \cite{Se}.

For topological spaces we adopt standard terminology. If $U$ is a subset of a topological space, then $\mathring{U}$ denotes the interior of $U$. As usual, for a Banach space $X$, the symbol $B_X$ stands for its closed unit ball.

For Banach spaces $X$ and $Y$, a bounded linear surjection $S:X\to Y$ is called a quotient map if the induced operator 
$\widehat{S}:X/\ker(S)\to Y$, $\widehat{S}(x+\ker S)=S(x)$, is a linear isometric isomorphism.

%--------------------------------------------------------------------------------------------------------------------------
%--------------------------------------------------------------------------------------------------------------------------
%--------------------------------------------------------------------------------------------------------------------------
%--------------------------------------------------------------------------------------------------------------------------
%--------------------------------------------------------------------------------------------------------------------------
%--------------------------------------------------------------------------------------------------------------------------

\section{The Lifting Theorem}
\label{SecLifting}

We begin this section with two basic results that will be employed in the proof of Theorem \ref{Main1}.

\begin{proposition}\label{NormEstimation}
If $S:X\to Y$ is a bounded linear operator and $T:Z\to Y$ is a quotient map, then  \[\|S\|=\inf\{\lambda>0:S[B_X]\subset T[\lambda \mathring{B}_Z]\}.\]
\end{proposition}
\begin{proof}
Let $q:Z\to Z/\ker T$ denote the canonical map. Since $T$ is a quotient map, the induced operator 
$\widehat{T}:Z/\ker T\to Y$ is an isometric isomorphism and $T=\widehat{T}\circ q$.

Thus,
\[
S[B_X]\subset T[\lambda \mathring{B}_Z]
\iff
S[B_X]\subset \widehat{T}[q(\lambda \mathring{B}_Z)]
\iff
\widehat{T}^{-1}\circ S[B_X]\subset q(\lambda \mathring{B}_Z).
\]

Since $q(\lambda \mathring{B}_Z)=\lambda \mathring{B}_{Z/\ker T}$, we obtain
\[
\inf\{\lambda>0:S[B_X]\subset T[\lambda \mathring{B}_Z]\}
=
\inf\{\lambda>0:\widehat{T}^{-1}\circ S[B_X]\subset \lambda \mathring{B}_{Z/\ker T}\}=\|\widehat{T}^{-1}\circ S\|.
\]
Since $\widehat{T}^{-1}$ is an isometry, we conclude that $\|\widehat{T}^{-1}\circ S\|=\|S\|$,
which completes the proof.
\end{proof}

\begin{proposition}\label{Prop:Fecho}
Let $X$ and $Y$ be Banach spaces. If $T:X\to Y$ is a surjective linear operator, then for every pair of real numbers $r<R$, we have $\overline{T[rB_X]}\subset T[R\mathring{B}_X]$.
\end{proposition}
\begin{proof}
It suffices to prove the case $r=1$. Consider the induced isomorphism $\widehat{T}:X/\ker T\to Y$, and let $\widehat{T}^{-1}$ be its inverse. Let $y \in \overline{T[B_X]}$ be arbitrary. There is a sequence $(x_n)_n$ in $B_X$ such that 
$T(x_n)\to y$. Then $x_n+\ker T=\widehat{T}^{-1}(Tx_n)\to \widehat{T}^{-1}(y)$ in $X/\ker T$. Since $\|x_n+\ker T\|_{X/\ker T}\leq \|x_n\|\leq 1$, it follows that $\|\widehat{T}^{-1}(y)\|_{X/\ker T}\leq 1$.
We write $\widehat{T}^{-1}(y)=x+\ker T$. Since 
\[\|\widehat{T}^{-1}(y)\|_{X/\ker T}=\|x+\ker T\|_{X/\ker T}=\inf_{z\in\ker T}\|x+z\|\leq 1,\]
and $R>1$, we may fix $z\in\ker T$ such that $\|x+z\|<R$. Therefore, $x+z\in R\mathring{B}_X$ and $T(x+z)=Tx=y$. We deduce that $\overline{T[B_X]}\subseteq T\bigl(R\mathring{B}_X\bigr)$.
\end{proof}

The following lemma is a well-known result that can be found at \cite[Theorem VI 7.1]{DunfordSchwartz}. 

\begin{lemma}\label{DiagramC(K)}
Let $L$ and $K$ be compact Hausdorff spaces and let $\mathcal{M}(K)$ be endowed with the weak$^*$-topology.
For every continuous function $g:L\to \mathcal{M}(K)$, there is a unique operator 
$S:C(K)\to C(L)$ such that $S^*(\delta_t)=g(t)$ for each $t\in L$ and $\|S\|=\sup_{t\in L}\|g(t)\|$.
Moreover, if $g$ is weakly continuous, then $S$ is weakly compact, and if
$g$ is norm-continuous, then $S$ is compact.
\end{lemma}

\begin{remark}
Conversely to Lemma \ref{DiagramC(K)}, given a bounded linear operator $S:C(K)\to C(L)$, the formula $t\mapsto S^*(\delta_t)$ defines 
a weak$^*$-continuous function $g:L\to \mathcal{M}(K)$ with $\|S\|=\sup_{t\in L}\|g(t)\|$.
\end{remark}

We are now in a position to prove our main result.

\begin{proof}[\textbf{Proof of Theorem \ref{Main1}}]
Let $S^*:\Lip(N)^*\to \Lip(M)^*$ be the dual operator and let $\epsilon>0$. 
Choose numbers $r,R$ such that $\|S\|<r<R<\|S\|+\epsilon$.

Since $\varPhi_M:\Lip(M)\to C(\beta \widetilde{M})$
is a linear isometric embedding, $\varPhi_M^*:\mathcal{M}(\beta \widetilde{M})\to \Lip(M)^*$ is a quotient map, and according to Proposition \ref{NormEstimation} we have 
\begin{equation}\label{eq:aux}
S^*[B_{\Lip(N)^*}]\subset \varPhi_M^*[r\mathring{B}_{\mathcal{M}(\beta \widetilde{M})}].
\end{equation}

We first prove the general case. We define a map $G:\widetilde{N}\to 2^{\mathcal{M}(\beta \widetilde{M})}$ by 
\[G(p,q)=\left\{\mu\in R B_{\mathcal{M}(\beta \widetilde{M})}: \varPhi_M^*(\mu)=S^{*}(\overline{m}_{pq})\right\}.\]

We observe that for every $(p,q)\in \widetilde{N}$, $G(p,q)$ is a nonempty closed and convex subset of $\mathcal{M}(\beta \widetilde{M})$. 
Indeed, the convexity is clear. To see that $G(p,q)$ is closed, note that $\varPhi_M^*$ is continuous, hence $G(p,q)$ is the inverse image of a singleton intersected with a closed ball. To prove nonemptiness, fix $(p,q)\in\widetilde{N}$. From \eqref{eq:aux} there exists $\mu_0\in r\mathring{B}_{\mathcal{M}(\beta \widetilde{M})}$ such that 
$\varPhi_M^*(\mu_0)=S^*(\overline{m}_{pq})$. In particular, $\mu_0\in G(p,q)$.

To see that $G$ is lower semi-continuous, let $U\subset \mathcal{M}(\beta \widetilde{M})$
be an arbitrary open subset. We prove that 
\[
G^{-1}[U]= \{(p,q)\in \widetilde{N}: G(p,q)\cap U\neq \emptyset\}
\]
is an open subset of $\widetilde{N}$. 

We start by noticing that, for all $(x,y),(p,q)\in\widetilde{N}$ and $f\in B_{\Lip(M)}$, the following relation holds true:
\begin{align*}
&|\langle S^*(\overline{m}_{xy})-S^*(\overline{m}_{pq}),f \rangle|= \left|\frac{S(f)(x)-S(f)(y)}{d(x,y)}-\frac{S(f)(p)-S(f)(q)}{d(p,q)}\right|\\
&=\frac{1}{d(x,y)d(p,q)}|(d(p,q)S(f)(x)-d(x,y)S(f)(p))+(d(x,y)S(f)(q)-d(p,q)S(f)(y))|\\
&\leq \frac{1}{d(x,y)d(p,q)}|d(p,q)(S(f)(x)-S(f)(p)) + S(f)(p)(d(p,q)-d(x,y))|\\
&+\frac{1}{d(x,y)d(p,q)}|d(p,q)(S(f)(q)-S(f)(y))+S(f)(q)(d(x,y)-d(p,q))|\\
&\leq \frac{\|S\|}{d(x,y)}(d(x,p)+d(y,q))+\frac{\|S\|}{d(x,y)d(p,q)}(d(p,0)+d(q,0))|d(p,q)-d(x,y)|.
\end{align*}

Since $f\in B_{\Lip(M)}$ is arbitrary, the previous relation implies that the function $h:\widetilde{N}\to \Lip(M)^*$ given by $h(p,q)=S^*(\overline{m}_{pq})$ is norm-continuous. 

Next, we claim that for every $(p,q)\in \widetilde{N}$ we have $G(p,q)\cap U\neq \emptyset$ if and only if $h(p,q)\in \varPhi_M^*[U\cap R\mathring{B}_{\mathcal{M}(\beta \widetilde{M})}]$. One implication is evident. For the reverse implication, given $(p_0,q_0)\in \widetilde{N}$, assume that $\mu\in G(p_0,q_0)\cap U$. If $\|\mu\|<R$, then $\mu\in U\cap R\mathring{B}_{\mathcal{M}(\beta \widetilde{M})}$. If $\|\mu\|= R$, by \eqref{eq:aux} we may choose $\mu_1\in r\mathring{B}_{\mathcal{M}(\beta \widetilde{M})}$ such that 
$\varPhi_M^*(\mu_1)=S^*(\overline{m}_{p_0q_0})$. For $\lambda\in(0,1)$ sufficiently close to $1$, define $\mu_\lambda=\lambda\mu+(1-\lambda)\mu_1$. Then $\mu_\lambda\in U\cap R\mathring{B}_{\mathcal{M}(\beta \widetilde{M})}$ and $\varPhi_M^*(\mu_\lambda)=S^*(\overline{m}_{p_0q_0})$. Hence, in any case there exists $\widetilde{\mu}\in U\cap R\mathring{B}_{\mathcal{M}(\beta \widetilde{M})}$ such that $h(p_0,q_0)=S^*(\overline{m}_{p_0q_0})=\varPhi_M^*(\widetilde{\mu})$ and this establishes our claim.

Now since $\varPhi_M^*$ is a quotient map, it is an open map; therefore,
$\varPhi_M^*[U\cap R\mathring{B}_{\mathcal{M}(\beta \widetilde{M})}]$ is an open subset of $\Lip(M)^*$. Hence
\[
G^{-1}[U]=h^{-1}[\varPhi_M^*[U\cap R\mathring{B}_{\mathcal{M}(\beta \widetilde{M})}]]
\]
is open.

By applying Michael's Selection Theorem \cite[Theorem 1.16]{BenLind}, the function 
$G$ admits a continuous selection $g:\widetilde{N}\to \mathcal{M}(\beta \widetilde{M})$. By considering 
$\mathcal{M}(\beta \widetilde{M})$ endowed with the weak$^*$-topology, and since $g[\widetilde{N}]$ is a norm-bounded subset, the function $g$ admits a unique extension $\overline{g}:\beta \widetilde{N}\to (\mathcal{M}(\beta \widetilde{M}),w^*)$ which, by Lemma \ref{DiagramC(K)}, induces an operator $\mathfrak{S}:C(\beta \widetilde{M})\to C(\beta \widetilde{N})$ satisfying 
\[
\|\mathfrak{S}\|=\sup\{\|\overline{g}(\xi)\|:\xi\in \beta\widetilde{N}\}\leq R<\|S\|+\epsilon
\]
and $\mathfrak{S}^*(\delta_{\xi})=\overline{g}(\xi)$ for every $\xi \in  \beta \widetilde{N}$.

Next we assume that $S$ is a compact operator. By Schauder's theorem \cite[Theorem 15.3]{Fabian}, $S^*$ is a compact operator. Since $\{\overline m_{pq}:(p,q)\in\widetilde N\} \subset B_{\Lip(N)^*}$, the set $K=\overline{\{S^*(\overline m_{pq}):(p,q)\in\widetilde N\}}^{\|\cdot\|}$
is norm-compact.

We define a map $F:K\to 2^{\mathcal{M}(\beta \widetilde{M})}$ by \[F(y)=\left\{\mu\in R B_{\mathcal M(\beta\widetilde M)}:\varPhi_M^*(\mu)=y\right\}.\] 

As in the general case, for every $y\in K$, $F(y)$ is a closed and convex subset of $\mathcal{M}(\beta \widetilde{M})$. Since $\varPhi_M^*$ is surjective, from \eqref{eq:aux} and Proposition \ref{Prop:Fecho} we obtain 
\begin{equation}\label{eq:aux2}
K\subset \overline{\varPhi_M^*\bigl[r\mathring{B}_{\mathcal{M}(\beta\widetilde M)}\bigr]}
\subset
\varPhi_M^*\bigl[R\mathring{B}_{\mathcal{M}(\beta\widetilde M)}\bigr]
\subset
\varPhi_M^*\bigl[RB_{\mathcal{M}(\beta\widetilde M)}\bigr].
\end{equation}
Hence $F(y)$ is also nonempty.

To establish that $F$ is lower semi-continuous, let $U\subset \mathcal{M}(\beta \widetilde{M})$
be an arbitrary open subset. We prove that 
\[
F^{-1}[U]=\{y\in K: F(y)\cap U\neq \emptyset\}=\varPhi_M^*[U\cap R\mathring{B}_{\mathcal{M}(\beta \widetilde{M})}]\cap K,
\]
which is therefore an open subset of $K$.

One inclusion is evident. For the reverse inclusion, let $y\in K$ be arbitrary and assume that $\mu\in F(y)\cap U$. If $\|\mu\|<R$, then $\mu\in U\cap R\mathring{B}_{\mathcal{M}(\beta \widetilde{M})}$, and we are done. If $\|\mu\|= R$, from \eqref{eq:aux2} we may choose $\mu_1\in R \mathring{B}_{\mathcal{M}(\beta \widetilde{M})}$ such that 
$\varPhi_M^*(\mu_1)=y$. Similarly as in the general case, we pick $\lambda\in(0,1)$ sufficiently close to $1$ and define $\mu_\lambda=\lambda\mu+(1-\lambda)\mu_1$. Then $\mu_\lambda\in U\cap R\mathring{B}_{\mathcal{M}(\beta \widetilde{M})}$ and $\varPhi_M^*(\mu_\lambda)=y$. Hence, in any case there exists $\widetilde{\mu}\in U\cap R\mathring{B}_{\mathcal{M}(\beta \widetilde{M})}$ such that $\varPhi_M^*(\widetilde{\mu})=y$. Therefore, $y\in \varPhi_M^*[U\cap R\mathring{B}_{\mathcal{M}(\beta \widetilde{M})}]\cap K$ and this establishes the reverse inclusion.

Once again, by applying Michael's Selection Theorem \cite[Theorem 1.16]{BenLind} (see also \cite{Michael1956}), we obtain a norm-continuous selection
$\sigma:K\longrightarrow\mathcal M(\beta\widetilde M)$.
The map
$h:\widetilde N\longrightarrow K$, given by
$h(p,q)=S^*(\overline m_{pq})$, is continuous, as we saw above, and therefore extends continuously to
$\overline h:\beta\widetilde N\longrightarrow K$.
Thus
\[
\overline g=\sigma\circ\overline h:
\beta\widetilde N\longrightarrow\mathcal M(\beta\widetilde M)
\]
is norm-continuous. As in the general case, the corresponding operator $\mathfrak{S}:C(\beta \widetilde{M})\to C(\beta \widetilde{N})$ satisfies 
\[
\|\mathfrak{S}\|=\sup\{\|\overline{g}(\xi)\|:\xi\in \beta\widetilde{N}\}\leq R<\|S\|+\epsilon
\]
and $\mathfrak{S}^*(\delta_{\xi})=\overline{g}(\xi)$ for every $\xi \in  \beta \widetilde{N}$. Since $\overline g$ is norm-continuous, by Lemma \ref{DiagramC(K)} $\mathfrak S$ is a compact operator.

Finally, to complete the proof, we note that in any of the cases above, for every $f\in \Lip(M)$ and $(p,q)\in \widetilde{N}$ we have
\begin{align*}
\mathfrak{S}(\varPhi_M(f))(p,q)&=\mathfrak{S}^*(\delta_{(p,q)})(\varPhi_M(f))=\int \varPhi_M(f) d g(p,q)\\
&=\langle \varPhi_M^*(g(p,q)),f \rangle=\langle S^*(\overline{m}_{pq}),f\rangle= \frac{S(f)(p)-S(f)(q)}{d(p,q)}=\varPhi_N(S(f))(p,q).
\end{align*}

\end{proof}

Theorem~\ref{Main1} and Lemma~\ref{DiagramC(K)} immediately yield the following representation.

\begin{corollary}
Let $M$ and $N$ be pointed metric spaces, let $S:\Lip(M)\to\Lip(N)$ be a bounded linear operator, and let $\epsilon>0$.
Then there exist a bounded linear operator $\mathfrak S:C(\beta\widetilde M)\to C(\beta\widetilde N)$
and a weak$^*$ continuous map $g:\beta\widetilde N\to\mathcal M(\beta\widetilde M)$ such that $\|\mathfrak S\|\le \|S\|+\epsilon$,
$\|g(\xi)\|\le \|S\|+\epsilon$ for every $\xi\in\beta\widetilde N$, and
\[
\mathfrak S(F)(\xi)
=
\int_{\beta\widetilde M} F\,dg(\xi)
\]
for every $F\in C(\beta\widetilde M)$ and every
$\xi\in\beta\widetilde N$. Moreover, $\mathfrak S(\Phi_M(f))=\Phi_N(S(f))$ for every $f\in\Lip(M)$. In particular, for every
$(p,q)\in\widetilde N$ and every $f\in\Lip(M)$,
\[
\frac{S(f)(p)-S(f)(q)}{d(p,q)}
=
\int_{\beta\widetilde M}\Phi_M(f)\,dg(p,q).
\]

If $S$ is a compact operator, then $\mathfrak S$ can be chosen compact and $g$ can be chosen norm-continuous.
\end{corollary}

%----------------------------------------------------------------------------------------------------------------------------
%----------------------------------------------------------------------------------------------------------------------------
%----------------------------------------------------------------------------------------------------------------------------
%----------------------------------------------------------------------------------------------------------------------------
%----------------------------------------------------------------------------------------------------------------------------
%----------------------------------------------------------------------------------------------------------------------------
%----------------------------------------------------------------------------------------------------------------------------
%----------------------------------------------------------------------------------------------------------------------------
%----------------------------------------------------------------------------------------------------------------------------

\section{Further Remarks}
\label{SecIso}

Theorem \ref{Main1} offers a different perspective on operators between Lipschitz spaces.  The image of the de Leeuw map can be understood as a subspace of the functions $u$ in $C(\beta \widetilde{M})$ satisfying the condition
\[
d(x,z)u(x,z)=d(x,y)u(x,y)+d(y,z)u(y,z)
\]
for all $x,y,z \in M$, under the convention that $u(x,x)=0$.

Motivated by the terminology from ergodic theory and groupoid theory (see, for example, \cite[Definition 3.1]{FHKP}), we shall refer to the identity above as the \emph{cocycle identity}. Note that, under this convention, we have $u(x,y)=-u(y,x)$ for every $(x,y)\in \widetilde M$.

For any metric space $M$, we consider
\[
\mathcal{Z}(M)
=
\{u\in C(\beta \widetilde M):u \text{ satisfies the cocycle identity}\}.
\]
This is clearly a subspace of $C(\beta \widetilde M)$, which we call the cocycle space of $M$.

\begin{proposition}
For every metric space $(M,d)$, the de Leeuw map is an isometric isomorphism from $\Lip(M)$ onto $\mathcal{Z}(M)$.
\end{proposition}

\begin{proof}
If $f\in \Lip(M)$, let $u_f=\Phi_M(f)$. It is straightforward to check that $u_f$ satisfies the cocycle identity. Conversely, let $u\in \mathcal{Z}(M)$ and let $o\in M$ be the distinguished point. Define $f:M\to \mathbb R$ by the formula $f(x)=d(x,o)u(x,o)$. Clearly $f(o)=0$. Moreover,
\[
f(x)-f(y)
=
d(x,o)u(x,o)-d(y,o)u(y,o)
=
d(x,y)u(x,y),
\]
where the last equality follows from the cocycle identity. Therefore,
\[
|f(x)-f(y)|
=
d(x,y)|u(x,y)|
\leq
\|u\|_\infty d(x,y)
\]
for every $x,y\in M$, and hence $f\in \Lip(M)$. It is clear that $\Phi_M(f)(x,y)=u(x,y)$ for every $(x,y)\in \widetilde{M}$. Thus $\Phi_M(\Lip(M))=\mathcal{Z}(M)$. Since the de Leeuw map is an isometry, the proof is complete.
\end{proof}

Theorem 1 may therefore be viewed as an extension theorem for operators acting on cocycle spaces. Conversely, any operator $T:C(\beta \widetilde M)\to C(\beta \widetilde N)$ satisfying $T(\mathcal{Z}(M))\subset \mathcal{Z}(N)$ induces, in a canonical way, an operator from $\Lip(M)$ into $\Lip(N)$.

Next, with regard to Theorem \ref{Main1}, the following question naturally arises.

\begin{problem}
Is it possible for a linear isomorphism from $\Lip(M)$ onto $\Lip(N)$ to be lifted to an isomorphism from $C(\beta\widetilde{M})$ onto $C(\beta\widetilde{N})$ in such a way that the above diagram commutes?
\end{problem}

If $(M,d_M)$ and $(N,d_N)$ are pointed metric spaces, $\gamma:N\to M$ is a bi-Lipschitz isomorphism such that $\gamma(0_N)=0_M$, and $r\neq 0$ is a real scalar, then the pair $(\gamma,r)$ induces a linear isomorphism $S:\Lip(M)\to \Lip(N)$ through the formula $S(f)(y)=r\cdot(f\circ \gamma)(y)$. This can be lifted to an isomorphism $\mathfrak{S}:C(\beta \widetilde{M})\to C(\beta \widetilde{N})$ given as follows: 
for each $F\in C(\beta \widetilde{M})$, $\mathfrak{S}(F)$ is the unique continuous extension to $\beta \widetilde{N}$ of the map 
\[\widetilde{N}\ni (x,y)\mapsto \frac{r\cdot d_M(\gamma(x),\gamma(y))}{d_N(x,y)}F(\gamma(x),\gamma(y))\in \R.\]

It is readily seen that $\mathfrak{S}(\varPhi_M(f))=\varPhi_N(S(f))$ for every $f\in \Lip(M)$. In particular, if $M$ and $N$ are uniformly concave \cite[Definition 3.33]{Weaver} and $S:\Lip(M)\to \Lip(N)$ is a linear isometric isomorphism, then $S$ is an operator of the above type; see \cite[Theorem 3.56]{Weaver}, and therefore admits an isometric lifting. The next example shows, however, that such isomorphic liftings cannot be obtained in general.

\begin{example}\label{ex01}Let $M=\N^\N$ and $N=\{0,1\}^\N$ both endowed with the metric
\begin{align*}d((x_n)_n,(y_n)_n)=\left\{
\begin{array}{rl}
0 & \text{if }(x_n)_n=(y_n)_n,\\
2^{-\min\{n \in \N:x_n\neq y_n\}} & \text{if }(x_n)_n\neq (y_n)_n.
\end{array} \right.
\end{align*}

with origin $0=(0,0,\ldots)$. Since $M$ and $N$ are separable ultrametric spaces, by the main result of \cite{CuthDoucha1}, the Banach spaces $\F(M)$ and $\F(N)$ are both linearly isomorphic to $\ell_1$. Hence, their respective dual spaces $\Lip(M)$ and $\Lip(N)$ are linearly isomorphic. However, since $M$ is not locally compact, $C(\beta \widetilde{M})$ admits a complemented copy of $\left(\bigoplus_{n \in \N} \ell_\infty\right)_{c_0}$; see \cite[Theorem 2]{Khmyleva}.
On the other hand, for the space $\widetilde{N}$ we may consider the sequence $(K_n)_n$, where $K_n=\{(a,b)\in N\times N : d(a,b) \geq 1/n\}.$
For each $n \in \N$ we have $K_n\subset \mathring{K}_{n+1}$, and since $N$ is compact, $K_n$ is a compact subset of $\widetilde{N}$.
It follows that $C(\beta \widetilde{N})$ has no complemented subspace isomorphic to $\left(\bigoplus_{n \in \N} \ell_\infty\right)_{c_0}$; see \cite[Theorem 1]{Khmyleva}. Consequently, the spaces $C(\beta \widetilde{M})$ and $C(\beta \widetilde{N})$ are not linearly isomorphic, whence no isomorphism $S:\Lip(M)\to \Lip(N)$ can be lifted to a linear isomorphism $\mathfrak{S}:C(\beta \widetilde{M})\to C(\beta \widetilde{N})$.
\end{example}

To conclude this section, we would like to leave open the following problem.

\begin{problem}
Let $M$ and $N$ be pointed metric spaces. Does every weakly compact operator
$S:\Lip(M)\to\Lip(N)$ admit a weakly compact lifting $\mathfrak S:C(\beta\widetilde M)\to C(\beta\widetilde N)$?
\end{problem}

One might try to adapt the proof of Theorem \ref{Main1} to the weakly compact setting. By Gantmacher's theorem \cite[Theorem 13.34]{Fabian}, the set
$K=\overline{\{S^*(\overline m_{pq}):(p,q)\in\widetilde N\}}^{\,w}$ is weakly compact and hence, when endowed with the relative weak topology, it is a compact Hausdorff space and therefore paracompact. Thus, in order to apply Michael's Selection Theorem, it would be enough to prove that, for every norm-open subset $U$ of $\mathcal M(\beta\widetilde M)$, the set
\[
\Phi_M^*\bigl[
U\cap R\mathring B_{\mathcal M(\beta\widetilde M)}
\bigr]\cap K
\]
is relatively weakly open in $K$. We do not know whether this property holds in general no even if $U$ is weakly open. Nevertheless, if the answer were affirmative, then the same argument used in the proof of Theorem \ref{Main1}, together with Proposition \ref{DiagramC(K)}, would yield a lifting theorem for weakly compact operators.

\section{Acknowledgements}

This research was supported by the Fundação de Amparo à Pesquisa do Estado de São Paulo (FAPESP), grant no.~2023/12916-1.

\end{document}